\documentclass[a4paper,11pt,leqno,oneside]{article}

\usepackage{amssymb,amsmath}
\usepackage[matrix]{xypic}

\newtheorem{lemma}{Lemma}[section] 
\newtheorem{definition}[lemma]{Definition} 
\newtheorem{theorem}[lemma]{Theorem}
\newtheorem{proposition}[lemma]{Proposition}
\newtheorem{prp-df}[lemma]{Proposition and Definition}
\newtheorem{corollary}[lemma]{Corollary} 

\newlength{\eqdemoffset}
\newlength{\enumrkoffset}
\newenvironment{proof}[1][Proof] {\begin{trivlist}\item[]{\sc #1.}~}
  {\ifhmode{\unskip~\nobreak}\else%
    {\nopagebreak\vspace{\eqdemoffset}\leavevmode\hfill}\fi%
    \hfill$\blacksquare$\end{trivlist}}

\newenvironment{rk}[1]{\refstepcounter{lemma}%
  \begin{trivlist}\item[]{\sc #1 \arabic{section}.\arabic{lemma}}~}%
  {\ifhmode{\unskip~\nobreak}\else%
    {\nopagebreak\vspace{\enumrkoffset}\leavevmode\hfill}\fi%
    \hfill$\square$\end{trivlist}}

\newcommand{\ds}{\displaystyle}
\newcommand{\ts}{\textstyle}
\newcommand{\Cat}{\mathcal{C}}
\newcommand{\Irr}{\mathrm{Irr}~}
\newcommand{\Mor}{\mathrm{Mor}~}
\newcommand{\Dom}{\mathrm{Dom}~}
\newcommand{\Tr}{\mathrm{Tr}~}
\renewcommand{\Re}{\mathrm{Re}~}
\renewcommand{\Im}{\mathrm{Im}~}
\newcommand{\Dir}{{\mathcal{D}}}
\renewcommand{\iff}{\textbf{iff} } 
\newcommand{\Cst}{$C^*$-\relax}
\newcommand{\tens}{\makebox[.8em][c]{$\otimes$}}
\newcommand{\tensexp}{\kern -.05em\otimes\kern -.05em}
\newcommand{\id}{{\rm id}}
\newcommand{\rond}{\makebox[1em][c]{$\circ$}}
\newcommand{\CC}{\mathbb C} 
\newcommand{\RR}{\mathbb R} 
\newcommand{\ZZ}{\mathbb Z}
\newcommand{\NN}{\mathbb N} 
\newcommand{\Ss}{\mathcal{S}}
\newcommand{\Hh}{\mathcal{H}}
\newcommand{\Ff}{\mathcal{F}}
\newcommand{\Cc}{\mathrm{Conv}}
\newcommand{\Tt}{\mathcal{T}}
\newcommand{\degree}{$^{\textrm{\footnotesize o}}$}

\begin{document}

\title{The Property of Rapid Decay \\ for Discrete Quantum Groups}
\author{Roland {\sc Vergnioux}}
\date{}

\maketitle

\begin{abstract}
  We introduce the Property of Rapid Decay for discrete quantum groups by equivalent
  characterizations that generalize the classical ones. We then investigate examples,
  proving in particular the Property of Rapid Decay for unimodular free quantum groups. We
  finally check that the applications to the $K$-theory of the reduced group \Cst algebras
  carry over to the quantum case. 
  \indent MSC~2000: 58B32 (46L80, 43A17, 46H30)
\end{abstract}

Let $\Gamma = F_N$ be the free group on $N$ generators, and denote by $l(\gamma)$ the
length of a reduced word $\gamma \in \Gamma$. In the founding paper \cite{Haagerup:MAP},
Haagerup proved that the norm of the reduced group \Cst algebra $C^*_r(\Gamma)$ can be
controlled by Sobolev $\ell^2$-norms associated to $l$:
\begin{displaymath}
  \forall x\in \CC\Gamma ~~ ||x||_{C^*_r\Gamma} \leq C\, ||x||_{2,s} := C\, 
  \left(\ts\sum_{\gamma\in\Gamma}\ds |(1+l(\gamma))^s x_\gamma|^2\right)^{1/2}
\end{displaymath}
for suitable constants $C$, $s>0$. This is remarkable since on the other hand the norm of
$C^*_r(\Gamma)$ always dominates the (non-weighted) $\ell^2$-norm on $\CC \Gamma$.
Moreover the norm of $C^*_r(\Gamma)$ is a non-trivial and very interesting data, whereas
the Sobolev norms are easily computable.

In fact this phenomenon occurs in many more cases and we say that a discrete
group $\Gamma$ endowed with a length function $l$ has the {\em Property of Rapid
  Decay} (Property RD) if the inequality above is satisfied on $\CC\Gamma$ for
some constants $C$, $s$. This general notion was introduced and studied by
Jolissaint in \cite{Jolissaint:RD_df}, where many examples are also presented.

Among the various applications of this property, let us mention the one
concerning $K$-theory, which is interesting from the point of view of
noncommutative geometry: thanks to the control on the norm of $C^*_r(\Gamma)$
provided by Property RD one can prove that the $K$-theory of $C^*_r(\Gamma)$
equals the $K$-theory of certain dense convolution subalgebras of rapidly
decreasing functions on $\Gamma$ \cite{Jolissaint:RD_K_th}. This fact was
notably used by V.~Lafforgue in his approach to the Baum-Connes conjecture via
Banach $KK$-theory \cite{Lafforgue:RD,Lafforgue:Baum-Connes}.

\bigskip

The foundations of the theory of discrete quantum groups are now very well
understood \cite{Woro:matrix,EffrosRuan:discrete,vanDaele:discrete,Woro:houches}, 
and it is a general motivation to figure out whether the classical operator
algebraic properties of discrete groups can also be useful in the quantum
framework. In this article we address this question for the Property of Rapid
Decay.

As we will see, there is a quite natural way to extend the definition of Property RD to
discrete quantum groups, and one can prove that the applications to $K$-theory still work
in the general case. The natural candidates for interesting quantum examples are the free
quantum groups of Wang-Banica \cite{DaeleWang:univ,Banic:U(n)}, and we will show that the
unimodular ones indeed have Property RD: this is the quantum analogue of the founding
result of Haagerup.

\bigskip

The first section of this article summarizes definitions and facts about discrete quantum
groups that are needed in the sequel. In the second section we give a definition of Property RD
for discrete quantum groups, establishing the equivalence between various characterizations
that generalize the classical ones.

In the third section, we investigate the main known classes of examples. Amenable discrete
quantum groups have Property RD \iff they have polynomial growth and this is the case of duals
of connected compact Lie groups, see Section~\ref{sec:amenable}. On the other hand it turns out
that Property RD implies unimodularity: this results from a necessary condition that we
introduce in Section~\ref{sec:non-unimod}. In Section~\ref{sec:free_quantum} we address the
case of the free quantum groups. The previous necessary condition is then sufficient --- in the
unitary case, this is an adaptation of the classical proof of Haagerup ---, and it is true in
the unimodular case --- this is the ``purely quantum'' part of the proof.

Finally we prove in the fourth section that the application to $K$-theory mentioned above still
holds in the quantum case. We deal with the case of the Banach algebras $\hat H_L^s$, with $s$
big enough, and also with $\hat H_L^\infty$, which is a smaller dense subalgebra but not a
Banach algebra anymore.

\bigskip

Let us conclude this introduction with a remark. In the original paper of Haagerup the Property
of Rapid Decay was used in conjunction with the fact that $l$ is conditionnally of negative
type to show that $C^*_r(F_N)$ has the Metric Approximation Property --- although it does not
have the Completely Positive Approximation Property since $F_N$ is not amenable. In the case of
the free quantum groups however one can show that the natural length, which satisfies Property
RD, is not conditionnally of negative type in the appropriate sense.

\section{Notation} \label{sec:notations}

In this section we introduce the notation and results about discrete quantum
group that are needed in the article. Having in mind the study of the Property
of Rapid Decay, it is natural to use the framework of Hopf \Cst algebras
\cite{Vallin:hopf}. The global situation to be considered consists in two
``dual'' Hopf \Cst algebras $(S,\delta)$ and $(\hat S,\hat \delta)$ related by a
multiplicative unitary $V\in M(\hat S\tens S)$ of discrete type. In fact this
discrete situation can be axiomatized equally at the level of $S$, $\hat S$ or
$V$:
\begin{itemize}
\item $(\hat S, \hat \delta)$ is a bisimplifiable unital Hopf \Cst algebra
  \cite{Woro:matrix,Woro:houches}.
\item $V$ is a regular multiplicative unitary with a unique co-fixed line
  \cite{BaajSkand:unit}.
\item $(S,\delta)$ is the completion of a multiplier Hopf $*$-algebra which is a 
  direct sum of matrix algebras \cite{vanDaele:multhopf,vanDaele:discrete}.
\end{itemize}
Notice also that these notions fit in the theory of locally compact quantum
groups \cite{KustVaes:lcqg}, yielding exactly the discrete case. Moreover, they
include many families of interesting examples, such as discrete groups, duals of
compact Lie groups and their $q$-deformations, unitary and orthogonal free
quantum groups. References and basic facts about these examples will be
presented along the text of the article as they are investigated.

To be more precise, and since it will be convenient for Property RD to have the
\Cst algebras $S$ and $\hat S$ represented on the same Hilbert space right from
the begining, let us start from the multiplicative unitary. It is a unitary
element of $B(H\tens H)$ for some Hilbert space $H$, such that
$V_{12}V_{13}V_{23} = V_{23}V_{12}$ on $H\tens H\tens H$.  Regularity means that
the norm closure of $(\id\tens B(H)_*)(\Sigma V)$ coincides with $K(H)$, where
$\Sigma$ is the flip operator, here on $H\tens H$. A co-fixed vector is an
element $e\in H$ such that $V(\zeta\tens e) = \zeta\tens e$ for all $\zeta\in
H$. The Hopf \Cst algebra $(S,\delta)$ can then be defined by
\begin{eqnarray*}
  &&S = \overline{\mathrm{Lin}} \{(\omega\tens\id)(V) ~|~ 
  \omega\in B(H)_*\} \text{~~~and} \\
  &&\forall a\in S~~~ \delta(a) = V (a\tens 1) V^* \text{.}
\end{eqnarray*} 
It follows from the work of Baaj, Skandalis \cite{BaajSkand:unit} and Podle\'s,
Woronowicz \cite{Woro:lorenz} that $(S, \delta)$ is indeed a Hopf \Cst algebra
which admits left and right Haar weights $h_L$, $h_R$.

\bigskip

The involutive monoidal category $\Cat$ of finite-dimensional unitary
representations of $S$ is a major tool when investigating in detail the
properties of $S$ and $\hat S$ \cite{Woro:dual}. We will denote by $H_\alpha$
the space of the representation $\alpha \in \Cat$. Choosing a complete set
$\Irr\Cat$ of representatives of the irreducible representations, we have
\begin{displaymath}
  S \simeq \bigoplus_{\alpha\in\Irr\Cat} L(H_\alpha) \text{.}
\end{displaymath}
We will denote by $\Ss$ the algebraic direct sum of the matrix algebras $L(H_\alpha)$, by
$p_\alpha \in S$ the central support of $\alpha\in\Irr\Cat$, and by $\Hh$ the algebraic direct
sum of the f.-d. subspaces $p_\alpha H$.

The category of f.-d. representations of $S$ can be endowed with a tensor product
$\alpha\tens\beta := (\alpha\tens\beta) \rond \delta$ and with a conjugation which we describe
now. Let $(e_i)$ be an orthonormal basis of $H_\alpha$. The conjugate object $\bar\alpha$ of
$\alpha\in \Cat$ is characterized, up to equivalence, by the existence of a conjugation map
$j_\alpha : H_\alpha \to H_{\bar\alpha}$, $\zeta \mapsto \bar\zeta$ which is an
anti-isomorphism such that ${t_\alpha : 1 \mapsto \sum e_i\tens \bar e_i}$ and $t'_\alpha :
\bar\zeta \tens \xi \mapsto (\zeta | \xi)$ are resp. elements of $\Mor (1, \alpha\tens \bar
\alpha)$ and $\Mor (\bar\alpha \tens\alpha, 1)$. If $\alpha$ is irreducible, the conjugation
map $j_\alpha$ is unique up to a scalar and one can renormalize it in such a way that $\Tr
j_\alpha^*j_\alpha = \Tr (j_\alpha^*j_\alpha)^{-1}$ and $j_{\bar\alpha} j_\alpha =\pm 1$.

The positive invertible maps $j_\alpha^*j_\alpha$ do not depend on the normalized
$j_\alpha$ and describe the non-trivial interplay between the involution of $\Cat$ and its
hilbertian structure. Putting them together, we obtain the modular element, a positive
unbounded multiplier $F \in S^\eta$ such that $p_\alpha F = j_\alpha^*j_\alpha$ for every
$\alpha\in\Irr\Cat$. One can show that $\delta(F) = F\tens F$, this is equivalent to the
fact that the restrictions of $j_{\alpha\tensexp\beta} = \Sigma \rond (j_\alpha\tens
j_\beta)$ to the irreducible subspaces of $H_\alpha\tens H_\beta$ are normalized
conjugation maps if $j_\alpha$, $j_\beta$ are so. 

The Haar weights of $(S,\delta)$ admit the following simple expressions in terms of $F$:
\begin{equation} \label{eq:expr_haar}
  \forall a\in p_\alpha S~~~ h_L(a) = m_\alpha \Tr (F^{-1}a) 
  \text{~~~and~~~} h_R(a) = m_\alpha \Tr (Fa) \text{,}
\end{equation}
where the positive number $m_\alpha = \Tr p_\alpha F = \Tr p_\alpha F^{-1}$ is called the
quantum dimension of $\alpha$. We say that $(S,\delta)$ is unimodular if $F = 1$.

The antipode $\kappa : \Ss \to \Ss$ is a linear and antimultiplicative map such that
$\alpha \rond \kappa \simeq \bar\alpha$ for all $\alpha\in\Cat$, and the co-unit
$\varepsilon\in S'$ equals the trivial representation $1_\Cat \in \Irr\Cat$. Let us recall
from \cite{Woro:lorenz} the following classical identities:
\begin{eqnarray} \label{eq:hopf_struct_coinv}
  &m \rond (\kappa \tens \id) \rond \delta = m \rond (\id \tens \kappa) \rond \delta 
  = 1_S \varepsilon \text{,}& \\ \label{eq:hopf_struct_counit}
  &(\varepsilon\tens\id) \rond  \delta = (\id\tens\varepsilon) \rond  \delta 
  = \id \text{~~~ and ~~~} \delta\rond\kappa = \sigma \rond (\kappa\tens\kappa) \rond
  \delta \text{,}& 
\end{eqnarray}
where $m : \Ss\tens\Ss \to \Ss$ denotes the product of $\Ss$ and $\sigma : \Ss\tens\Ss \to
\Ss\tens\Ss$ is the flip map.

\bigskip

The reduced dual Hopf \Cst algebra of $(S,\delta)$ will play an important role
in this paper. It is defined from the left leg of the multiplicative unitary $V$
under consideration:
\begin{eqnarray*}
  &&\hat S = \overline{\mathrm{Lin}} \{(\id\tens\omega)(V) ~|~ 
  \omega\in B(H)_*\} \text{~~~and} \\
  &&\forall \hat a\in \hat S~~~ \hat\delta(\hat a) = 
  V^* (1\tens \hat a) V \text{.}
\end{eqnarray*}
The pair $(\hat S, \hat\delta)$ is then the (unital) Hopf \Cst algebra of a compact quantum
group \cite{Woro:houches}. In particular it admits a two-sided Haar state $\hat h$ which is in
our setting the vector state associated to any co-fixed unit vector $e\in H$.  We say that
$(S,\delta)$ is amenable if $(\hat S,\hat\delta)$ admits a continuous co-unit.

We will consider the Fourier transform defined on the dense subspace $\Ss \subset S$ in
the following way:
\begin{displaymath}
  \forall a\in\Ss ~~~ \Ff(a) = (\id\tens h_R)(V^*(1\tens a)) \in \hat S \text{,}
\end{displaymath}
and we denote by $\hat\Ss$ the image of $\Ff$, a dense subspace of $\hat S$.
Observe that other choices are possible for the definition of $\Ff$. One can
check that our choice makes $\Ff$ isometric with respect to the GNS norms
associated to $h_R$ and $\hat h$. More precisely, let $e$ be a co-fixed unit
vector for $V$ and put $\Lambda(a) = \Ff(a)e$ for $a\in\Ss$, this defines a GNS
map $\Lambda$ for $h_R$ such that $\Lambda(ab) = a \Lambda(b)$ and $V
(\Lambda\tens\Lambda) (a\tens b) = (\Lambda\tens\Lambda) (\delta(a)(1\tens b))$
for any $a$, $b\in\Ss$.

\bigskip

All the unbounded operators we have used in this introduction are of an almost
trivial kind: they are self-adjoint and affiliated to $S$ or $S\tens S$. Due to
the very special structure of $S$, operators affiliated to $S$, which are closed
by definition, admit $\Hh$ as a core and they form a {$*$-algebra} $\Ss^\eta$
which identifies with the algebra of (algebraic) multipliers of the Pedersen
ideal $\Ss$.  Moreover in this case symmetry implies self-adjointness. Notice
that one can apply $*$-homomorphisms such as $\delta$ and $\varepsilon$ to
elements of $S^\eta$, but also proper maps such as the antipode $\kappa$
\cite{Woro:lorenz}.

\section{The Property of Rapid Decay}

\subsection{Lengths}

Before introducing the Property of Rapid Decay it is necessary to discuss the notion of length
for discrete quantum groups. In particular we establish at Lemmas~\ref{lem:domin}
and~\ref{lem:triangles} some elementary properties of these lengths.

\begin{definition} \label{df:length}
  Let $(S,\delta)$ be the Hopf \Cst algebra of a discrete quantum group. A length on
  $(S,\delta)$ is an unbounded multiplier $L\in S^\eta$ such that $L\geq 0$,
  $\varepsilon(L) = 0$, $\kappa (L)_{|\Hh} = L_{|\Hh}$ and $\delta(L) \leq 1\tens L +
  L\tens 1$.  Given such a length, we denote by $p_n \in M(S)$ the spectral projection of
  $L$ associated to the interval $[n, n+1[$, for $n\in\NN$.
\end{definition}

\begin{rk}{Example} \label{rk:central}
  \begin{enumerate}
  \item \label{item:central} A complex-valued function $l$ on $\Irr\Cat$ will be
    called a length function if we have for any $\alpha$, $\beta$, $\beta' \in
    \Irr\Cat$
    \begin{eqnarray*}
      &l(\alpha) \geq 0 \text{,~~} l(1_\Cat) = 0 \text{,~~}
      l(\bar\alpha) = l(\alpha) \text{~~and}& \\
      &\alpha\subset \beta\tens\beta' ~~ \Longrightarrow
      l(\alpha) \leq l(\beta) + l(\beta') \text{.}&
    \end{eqnarray*}
    It is easy to observe that $L = \sum l(\alpha)p_\alpha$ is then a central length on
    $(S,\delta)$, and that all of these are obtained in this way. Notice that $\alpha
    \subset \beta\tens\beta'$ $\Longleftrightarrow$ $(p_\beta \tens p_{\beta'})
    \delta(p_\alpha) \neq 0$, by definition of the representation $\beta\tens\beta'$.
    
  \item \label{item:finite_gen} We say that $(S,\delta)$ is finitely generated if there exists
    a finite subset $\Dir = \bar\Dir \subset \Irr\Cat$, not containing $1_\Cat$, such that any
    element of $\Irr\Cat$ is contained in a multiple tensor product of elements of $\Dir$. The
    distance to the origin in the classical Cayley graph associated to $(\hat S, \Dir)$, in the
    sense of \cite{Vergnioux:cayley}, defines then a length function on $\Irr\Cat$ with values
    in $\NN$ given by
    \begin{displaymath}
      l(\alpha) = \min \{k\in\NN ~|~ \exists \beta_1, \ldots,
      \beta_k \in \Dir ~~~ \alpha \subset \beta_1 \tens \cdots \tens
      \beta_k \} \text{.}
    \end{displaymath}
    The corresponding central length $L_0$ on $(S,\delta)$ is called the word length of
    $(S,\delta)$ with respect to $\Dir$.
  \end{enumerate}
\end{rk}

\begin{lemma} \label{lem:domin}
  Let $L_0$ and $L$ be lengths on $(S,\delta)$, and assume that $L_0$ is a word
  length. Then there exists $\epsilon>0$ such that $L_0 \geq \epsilon L$.
\end{lemma}

\begin{proof}
  Using the spectral projections of $L_0$, put $p=p_0 + p_1 \in Z(\Ss)$ and $C =
  ||Lp||$. For any $n\in\NN^*$ we have then
  \begin{displaymath}
    p^{\tensexp n}\delta^{n-1}(L) \leq p^{\tensexp n} (L\tens 1\tens \cdots + 
    1\tens L\tens 1\tens\cdots + \cdots + 1\tens1\tens\cdots\tens L)  \leq Cn \text{.}
  \end{displaymath}
  Take $\alpha\in\Irr\Cat$ such that $l(\alpha) = n$ : we have $p^{\tensexp n}
  \delta^{n-1} (p_\alpha) \neq 0$. The \Cst algebra $p_\alpha S$ being simple,
  this implies that the restriction of $p^{\tensexp n}\delta^{n-1}$ to $p_\alpha
  S$ is an isometry. In particular we can write
  \begin{eqnarray*}
    L p_\alpha \leq ||L p_\alpha|| p_\alpha &=& 
    ||p^{\tensexp n} \delta^{n-1} (Lp_\alpha)|| p_\alpha \\ &\leq&
     ||p^{\tensexp n} \delta^{n-1} (L)|| p_\alpha \leq C n p_\alpha 
     = C L_0 p_\alpha \text{.}
  \end{eqnarray*}
  Since this holds for any minimal central projection $p_\alpha$, the desired result is
  proved with $\epsilon = 1/C$. 
\end{proof}

\begin{lemma} \label{lem:triangles}
  Let $L$ be a central length on $(S,\delta)$. Denote by $\Tt \subset \NN^3$ the
  set of triples $(k,l,n)$ such that $\delta(p_n) (p_k\tens p_l) \neq 0$.
  \begin{enumerate}
  \item \label{item:lem_it} $\delta(L) \geq 1\tens L - L\tens 1$ and $\delta(L)
    \geq L\tens 1-1\tens L$.
  \item \label{item:lem_tri_sym} $\Tt$ is stable under permutation and contains
    $(n,n,0)$ for any $n\in\NN$.
  \item \label{item:lem_tri_it} $(k,l,n) \in \Tt$ $\Longrightarrow$ $|k-l|-2 \leq n \leq
    |k+l|+2$.
  \end{enumerate}
\end{lemma}

\begin{proof}
  \ref{item:lem_it}. Let $l$ be the length function on $\Irr \Cat$ corresponding
  to $L$ like in Example~\ref{rk:central}.\ref{item:central}. For the first
  inequality it is enough to prove that, for any inclusion $\alpha \subset
  \beta\tens\beta'$:
  \begin{eqnarray*}
    (p_\beta\tens p_{\beta'}) \delta(p_\alpha) \delta(L) &\geq& 
    (p_\beta\tens p_{\beta'}) \delta(p_\alpha) (1\tens L - L\tens 1) \\
    \Longleftrightarrow ~~ l(\alpha) &\geq& l(\beta') - l(\beta).
  \end{eqnarray*}
  This results from the equivalence $\alpha \subset \beta\tens\beta'$
  $\Longleftrightarrow$ $\beta' \subset \bar\beta\tens\alpha$ and the fact that
  $l$ is a length function. Similarly the second inequality results from the
  other equivalent inclusion $\beta\subset \alpha\tens\bar\beta'$.

  \ref{item:lem_tri_sym}. First notice that the projections $p_n$ are central. Using the
  identity~(\ref{eq:hopf_struct_counit}) and the fact that $\kappa(p_n) = p_n$, we can
  write
  \begin{eqnarray*}
    \delta(p_n) (p_k\tens p_l) &=& \delta(\kappa(p_n)) (p_k\tens p_l) = 
    \sigma(\kappa\tens\kappa)\delta(p_n) \,\cdot\, (p_k\tens p_l) \\ &=& 
    \sigma(\kappa\tens\kappa) \left((p_l\tens p_k)\delta(p_n)\right) \text{,}
  \end{eqnarray*}
  hence $(k,l,n) \in \Tt \Longrightarrow (l,k,n)\in\Tt$. On the other hand,
  using the identities (\ref{eq:hopf_struct_coinv}) and
  (\ref{eq:hopf_struct_counit}) we have
  \begin{eqnarray*}
    && (m\tens\id)(\kappa\tens\id\tens\id)(\id\tens\delta) ~ (\delta(p_n)(p_k\tens p_l)) = \\
    && \makebox[2cm]{} = (m\tens\id)(\kappa\tens\id\tens\id) ~ 
    (\delta^2(p_n) (p_k\tens\delta(p_l))) \\
    && \makebox[2cm]{} = (p_k\tens 1) ~ [(m\tens\id)(\kappa\tens\id\tens\id)\delta^2(p_n)] ~ 
    \delta(p_l) \\
    && \makebox[2cm]{} = (p_k\tens 1) (1\tens p_n) \delta(p_l)
  \end{eqnarray*}
  hence $(k,l,n) \in \Tt \Longrightarrow (k,n,l)\in\Tt$.
  
  \ref{item:lem_tri_it}. By definition one has $np_n\leq Lp_n \leq (n+1)p_n$, so that
  \begin{eqnarray*}
    \delta(L)\delta(p_n) (p_k\tens p_l) &\geq& n \delta(p_n) (p_k\tens p_l) \text{~~~and}
    \\ \delta(L)\delta(p_n) (p_k\tens p_l) &\leq& 
    (L\tens 1 + 1\tens L )\delta(p_n) (p_k\tens p_l) \\
    &\leq& (k+l+2) \delta(p_n) (p_k\tens p_l) \text{.}
  \end{eqnarray*}
  Now if $(k,l,n) \in \Tt$, the projection $\delta(p_n) (p_k\tens p_l)$ is non-zero and
  the above inequalities show that $n\leq k+l+2$. Applying this result to $(k,n,l)$ and
  $(n,l,k)$ yields the two other inequalities of the statement. 
\end{proof}

\subsection{Definition}

We now introduce the Property of Rapid Decay for discrete quantum groups with respect to a
central length. Like in the classical case, it is mainly about controlling the norm of the
reduced dual \Cst algebra by Sobolev norms, which are much simpler. If $L$ is a length on
$(S,\delta)$ and $s\in\RR_+$ we will use the following notation, where $e\in H$ is a co-fixed
unit vector:
\begin{eqnarray*}
  \forall a\in\Ss && ||a||_2^2 := h_R(a^*a) \text{,~~~}
  ||a||_{2,s} := ||(1+L)^s a||_2 \text{~~~ and} \\
  \forall \hat a\in\hat\Ss && ||\hat a||_2^2 := \hat h(\hat a^*\hat a)
  =||\hat a e||^2 \text{,~~~}
  ||\hat a||_{2,s} := ||(1+L)^s \hat a e|| \text{.}
\end{eqnarray*}
We will also denote by $H_L^s$ (resp. $\hat H_L^s$) the completion of $\Ss$ (resp.
$\hat\Ss$) with respect to the $2,s$-norm. We have the norm inequalities $||a||_2 \leq
||a||_{2,s}$ and $||a||_2 = ||\Ff(a)||_2 \leq ||\Ff(a)||$ and hence the following
continuous embeddings:
\begin{displaymath}
  \xymatrix@R-3ex{
    \raisebox{0pt}[2.5ex][1ex]{$H_L^s~$} \ar@{=}[d]^\Ff \ar@{^(->}[r] & 
    \raisebox{0pt}[2.5ex][1ex]{$~H~$} \ar@{=}[d]^\Ff \ar@{^(->}[r] & 
    \raisebox{0pt}[2.5ex][1ex]{$~S$} \\ 
    \raisebox{0pt}[2.5ex][1ex]{$\hat H_L^s~$} \ar@{^(->}[r] & 
    \raisebox{0pt}[2.5ex][1ex]{$~\hat H~$} & 
    \raisebox{0pt}[2.5ex][1ex]{$~\hat S$} \ar@{_(->}[l] }
\end{displaymath}

\begin{prp-df} \label{df:RD}
  Let $L$ be a central length on the Hopf \Cst algebra $(S, \delta)$ of a discrete quantum
  group. We say that $(S,\delta,L)$ has the Property of Rapid Decay (Property RD) if one
  of the following equivalent conditions is satisfied:
  \begin{enumerate}
  \item \label{item:df_conv} $\exists~ C, s \in \RR_+ ~~ \forall a\in\Ss ~~~ ||\Ff(a)||
    \leq C ||a||_{2,s}$
  \item \label{item:df_dual} $\exists~ C, s \in \RR_+ ~~ \forall \hat a\in\hat\Ss ~~~
    ||\hat a|| \leq C ||\hat a||_{2,s}$
  \item \label{item:df_rapid} $\hat H_L^\infty := \bigcap_{s\geq 0} \hat H_L^s \subset
    \hat S$ (as subspaces of $H$)
  \item \label{item:df_poly} $\exists~ P \in \RR[X] ~~ \forall n\in\NN,~ a\in p_n\Ss ~~~
    ||\Ff(a)|| \leq P(n) ||a||_2$
  \item \label{item:df_bloc} $\exists~ P \in \RR[X] ~~ \forall n\in\NN,~ a\in p_n\Ss ~~
    \forall k, l \in\NN ~~~ ||p_l \Ff(a) p_k|| \leq P(n) ||a||_2$
  \end{enumerate}
  We say that $(S, \delta)$ has Property RD if there exists a central length on it
  satisfying one of these conditions.
\end{prp-df}

\begin{proof}
  \ref{item:df_conv} $\Longleftrightarrow$ \ref{item:df_dual} is clear because $\Ff$ is a
  bijection between $\Ss$ and $\hat\Ss$ such that $\Lambda(a) = \Ff(a)e$.
  \ref{item:df_dual} $\Longrightarrow$ \ref{item:df_rapid} is immediate whereas
  \ref{item:df_rapid} $\Longrightarrow$ \ref{item:df_dual} follows, like in the classical
  case, from the closed graph Theorem \cite[{\S}3,~cor.~5]{Bourbaki:vect_top_val} for the
  inclusion of the Fr\'echet space $\hat H_L^\infty$ into $\hat S$. Note that the family
  of closed balls centered in $0$ with respect to all the norms $||\,\cdot\,||_{2,s}$
  forms a fundamental system of neighbourhoods of $0$ in $\hat H_L^\infty$.
  
  \ref{item:df_conv} $\Longrightarrow$ \ref{item:df_poly} $\Longrightarrow$
  \ref{item:df_bloc} are evident. For \ref{item:df_poly} $\Longrightarrow$
  \ref{item:df_conv}, choose $C, s \in \RR_+$ such that $P(n) \leq C (n+1)^s$ for all
  $n\in\NN$. We have then $||\Ff(a)|| \leq C \sum (n+1)^s ||p_na||_2$ and, like in the
  classical case \cite[lemma~1.5]{Haagerup:MAP}, we use the Cauchy-Schwartz inequality and
  the fact that $(n+1)^{s+1}||p_na||_2 \leq ||(L+1)^{s+1}p_n a||_2$ to conclude that
  $||\Ff(a)|| \leq C' ||a||_{2,s+1}$.
  
  For \ref{item:df_bloc} $\Longrightarrow$ \ref{item:df_poly} we will follow the proof of
  \cite[lemma~1.4]{Haagerup:MAP}, using the results of Lemma~\ref{lem:triangles}. Let $a$
  be an element of $p_n\Ss$. According to the following computation, we have $p_l \Ff(a)
  p_k \neq 0$ $\Longrightarrow$ $(k,l,n) \in \Tt$:
  \begin{eqnarray*}
    p_l \Ff(a) p_k &=& (\id\tens h_R) ((p_l\tens 1)V^*(p_k\tens a)) \\
    &=& (\id\tens h_R) (V^* \delta(p_l) (p_k\tens p_na)) \text{.}
  \end{eqnarray*}
  As a result we can write, for any $\xi\in H$:
  \begin{eqnarray*}
    ||\Ff(a)\xi||^2 &=& \sum_l ||p_l\Ff(a)\xi||^2 \leq 
    \sum_l (\ts\sum_k\ds ||p_l\Ff(a)p_k\xi||)^2 \\
    &\leq& P(n)^2 ||a||_2^2 \sum_l (\ts\sum_{(k,l,n)\in\Tt}\ds ||p_k\xi||)^2 \text{.}
  \end{eqnarray*}
  
  Moreover by Lemma~\ref{lem:triangles} the cardinal of $\{p ~|~ (n,q,p)\in\Tt\}$ is
  bounded above by $2 \min (q,n) + 5 \leq 2n+5$. Using this estimate twice and the
  Cauchy-Schwartz inequality we obtain
  \begin{eqnarray*}
    \sum_l (\ts\sum_{(k,l,n)\in\Tt}\ds ||p_k\xi||)^2 &\leq&
    \sum_l (2n+5) (\ts\sum_{(k,l,n)\in\Tt}\ds ||p_k\xi||^2) \\
    &\leq& (2n+5)^2 \ts\sum_k\ds ||p_k\xi||^2 = (2n+5)^2 ||\xi||^2 \text{.}
  \end{eqnarray*}
  Finally $||\Ff(a)|| \leq (2n+5) P(n) ||a||_2$ and Condition~\ref{item:df_poly} is
  satisfied. 
\end{proof}

\begin{rk}{Remark}
  Conditions \ref{item:df_conv}--\ref{item:df_poly} are still equivalent if $L$ is a
  non-central length on $(S,\delta)$ and hence they can be used to define a Property RD with
  respect to non-central lengths. Besides, if $L'$ is a central length on $(S,\delta)$ such
  that $L' \geq \epsilon L$ for some $\epsilon > 0$, it is easy to check that the Property RD
  for $L$ implies the Property RD for $L'$. Now, assume that $(S, \delta)$ is finitely
  generated and choose a word length $L_0$ on it. The preceeding observations and
  Lemma~\ref{lem:domin} show that $(S,\delta)$ has Property RD, possibly with respect to a
  non-central length, if and only if $(S,\delta,L_0)$ has Property RD. As a result, the use of
  central lengths for the study of Property RD is not a restriction in the finitely generated
  case.
\end{rk}

\begin{rk}{Example}
  When the \Cst algebra $S$ is commutative, it is of the form $c_0(\Gamma)$ with
  $\Gamma$ a discrete group, and the coproduct $\delta$ is given by the formula
  $\delta(f) (\alpha, \beta) = f(\alpha\beta)$. Then the notions of length and
  of Property RD studied in this paper coincide with the classical notions
  introduced by Jolissaint~\cite{Jolissaint:RD_df} after the founding paper of
  Haagerup~\cite{Haagerup:MAP} about the convolution algebras of the free group.
  For a recent account on Property RD for discrete groups, including examples,
  counter-examples and more references, we refer the reader to~\cite[chapter
  7]{Valette:livre_BC}.
\end{rk}

\section{Special cases and quantum examples}
\label{sec:examples}

\subsection{The amenable case}
\label{sec:amenable}

For an amenable discrete group, Property RD is equivalent to polynomial growth
\cite[cor.~3.1.8]{Jolissaint:RD_df}. In this section we extend this result to the case of
discrete quantum groups. The main motivation is the study of the duals of compact Lie
groups and their $q$-deformations: as a matter of fact these discrete quantum groups are
all amenable \cite[cor.~6.2]{Banica:CQG_subfactors}.

\begin{definition}
  Let $L$ be a central length on the Hopf \Cst algebra $(S,\delta)$ of a discrete quantum
  group. Denote by $(p_n)$ the corresponding sequence of spectral projections. We say that
  $(S,\delta,L)$ has polynomial growth if
  \begin{displaymath}
    \exists~ P \in \RR[X] ~~ \forall n\in\NN ~~~ h_R(p_n) \leq P(n) \text{.}
  \end{displaymath}
\end{definition}

\begin{rk}{Remark}
  Let $l$ be the length function on $\Irr\Cat$ associated to $L$, and recall
  that we denote by $m_\alpha$ the quantum dimension of $\alpha \in \Irr\Cat$.
  From the expression of $h_R$ given in Section~\ref{sec:notations} we see that
  \begin{displaymath}
    h_R(p_n) = \sum \{m_\alpha^2 ~|~ l(\alpha)\in [n,n+1[\} \in [0, +\infty] \text{.}
  \end{displaymath}
  In particular this implies that $h_R(p_n) = h_L(p_n)$. Moreover, put 
  \begin{eqnarray*}
    && s_n := \mathrm{Card} \{\alpha\in\Irr\Cat ~|~ l(\alpha)\in [n,n+1[\} \text{~~~and} \\
    && M_n := \sup \{m_\alpha ~|~ l(\alpha)\in [n,n+1[\} \text{.}
  \end{eqnarray*}
  The previous expression shows that $(S,\delta,L)$ has polynomial
  growth \iff the sequences $(s_n)$ and $(M_n)$ have polynomial
  growth. In the case of a discrete group, $M_n = 1$ for each $n$ and
  one recovers the classical notion of polynomial growth with respect
  to $l$.
\end{rk}

\begin{lemma} \label{lem:growth_non_unimod}
  Let $L$ be a central length on the Hopf \Cst algebra $(S,\delta)$ of a non-unimodular
  discrete quantum group. Then $(S,\delta,L)$ does not have polynomial growth. Moreover if $L$
  is a word length the sequences $(||p_nF||)_n$ and $(||p_nF^{-1}||)_n$ grow geometrically.
\end{lemma}

\begin{proof}
  Let us first notice that $L$ can be assumed to be a word length, even for the first
  assertion.  As a matter of fact, the non-unimodularity implies the existence of an
  irreducible representation $\alpha\in\Irr\Cat$ such that $p_\alpha F\neq p_\alpha$. By
  restricting to the ``subgroup generated by $\alpha$'', see \cite[section~2]{Vergnioux:free},
  one can assume that $\Dir = \{\alpha, \bar\alpha\}$ generates $\Cat$. By
  Lemma~\ref{lem:domin} one can moreover assume that $L$ is the word length associated to
  $\Dir$: as a matter of fact if $\epsilon L' \leq L$ with $L'$ having polynomial growth, then
  $L$ has polynomial growth.
  
  Now let $\Dir$ be the generating subset defining the word length $L$. The equality $\Tr
  p_\alpha F = \Tr p_\alpha F^{-1}$ shows that the greatest eigenvalue of $p_\alpha F$ is
  greater than or equal to $1$, with equality \iff $p_\alpha F = p_\alpha$. Let $\alpha$ be
  the element of $\Dir$ such that $p_\alpha F$ has the greatest eigenvalue $\lambda$.
  Because $\delta(F) = F\tens F$, the $n$-th power $\lambda^n$ is an eigenvalue of
  \begin{displaymath}
    p_\alpha^{\tensexp n}\delta^{n-1}(F) = \ts\sum\ds \{ p_\alpha^{\tensexp
      n}\delta^{n-1}(p_\beta F) ~|~ \beta \subset \alpha^{\tensexp n}\} \text{.}
  \end{displaymath}
  Since the maps $p_\alpha^{\tensexp n}\delta^{n-1}(p_\beta \,\cdot\,)$ are injective
  $*$-homomorphisms which have pairwise orthogonal ranges when $\beta$ varies, we conclude that
  one of the $p_\beta F$ admits $\lambda^n$ as an eigenvalue. In the same way one sees that the
  eigenvalues of $p_\beta F$, with $\beta \subset \alpha^{\tensexp n-1}$, are less than
  $\lambda^{n-1}$. As a result $||p_n F|| = \lambda^n$ and $\lambda > 1$ by non-unimodularity.
  In particular $(S,\delta,L)$ does not have polynomial growth. One proceeds in the same way
  for $||p_n F^{-1}||$, using a minimal eigenvalue of the $p_\alpha F$, $\alpha\in\Dir$.
\end{proof}

\pagebreak[3]
\begin{proposition} \label{prp:growth}
  Let $L$ be a central length on the Hopf \Cst algebra $(S,\delta)$ of a discrete quantum
  group.
  \begin{enumerate}
  \item \label{item:prp_growth_nec} If $(S,\delta)$ is amenable and $(S,\delta,L)$ has
    Property RD, then $(S,\delta,L)$ has polynomial growth.
  \item \label{item:prp_growth_suf} If $(S,\delta,L)$ has polynomial growth, then
    $(S,\delta,L)$ has Property RD.
  \end{enumerate}
\end{proposition}

\begin{proof}
  \ref{item:prp_growth_nec}. By hypothesis there exists a continuous co-unit
  $\hat\varepsilon$ on $(\hat S,\hat\delta)$. Let us show that $x :=
  (\hat\varepsilon\tens\id) (V)$ equals the identity: one has
  \begin{displaymath}
    x^2 = (\hat\varepsilon\tens\hat\varepsilon\tens\id) (V_{13}V_{23})
    = (\hat\varepsilon\tens\hat\varepsilon\tens\id) (\hat\delta\tens\id) (V)
    = (\hat\varepsilon\tens\id) (V) = x \text{,}
  \end{displaymath}
  but on the other hand $x$ is unitary because $\hat\varepsilon$ is a $*$-character.
  Hence $(\hat\varepsilon\tens\id) (V) = \id$ and, by definition of $\Ff$,
  $\hat\varepsilon \rond \Ff = h_R$. As a result we can write, like in the classical case:
  \begin{displaymath}
    \forall a\in\Ss ~~~ |h_R(a)| = |\hat\varepsilon\Ff(a)| 
    \leq ||\Ff(a)|| \leq C ||a||_{2,s} \text{,}
  \end{displaymath}
  for the constants $C$, $s\in\RR_+$ given by Property RD. Applying this to the
  projections $p_n$ gives the desired result:
  \begin{displaymath}
    h_R(p_n) \leq C ||(1+L)^s p_n||_2 \leq C(2+n)^s \sqrt{h_R(p_n)},
  \end{displaymath}
  hence $h_R(p_n) \leq C^2 (2+n)^{2s}$ for all $n\in\NN$.
  
  \ref{item:prp_growth_suf}. Let $a\in p_n\Ss$, in particular we have $a = pa$ for some central
  projection $p\in S$. We consider the linear functional $ah_R := h_R(\,\cdot\, a)$
  defined on the \Cst algebra $S$ and we first assume that it is positive. One can then
  write
  \begin{displaymath}
    ||\Ff(a)|| = ||(\id\tens ah_R)(V^*)|| \leq ||ah_R|| \text{.}
  \end{displaymath}
  Let $(u_i)$ be an approximate unit of $S$, we have $ah_R(u_i) = ah_R(u_ip)$ and, by
  taking limits, $||ah_R|| = ah_R(p) = h_R(a)$. As a result
  \begin{displaymath}
    ||\Ff(a)|| \leq h_R(a) = h_R(p_na) \leq \left( h_R(p_n) h_R(a^*a) \right)^{1/2} \leq
    \sqrt{P(n)}~ ||a||_2 \text{.} 
  \end{displaymath}
  
  The general case follows by decomposing $a\in p_n\Ss$ into a linear
  combination $\sum_{k=0}^3 i^k a_k$ of $4$ positive elements: the negative and
  positive parts of $\Im a$ and $\Re a$.  Since $(S,\delta)$ is necessarily
  unimodular according to Lemma~\ref{lem:growth_non_unimod}, $h_R$ is a trace
  and one can check that $||a||_2 = (\sum ||a_k||_2^2)^{1/2}$, which is greater
  than $(\sum ||a_k||)/2$.
\end{proof}

\begin{rk}{Example}
  Let $G$ be a compact group and take $S = C^*(G)$, $\delta(U_g) = U_g\tens U_g$. Then
  $(S,\delta)$ is the Hopf \Cst algebra of a discrete quantum group which is called the
  dual of $G$. By definition, $\Cat$ identifies in this case with the category of the
  finite-dimensional unitary representations of $G$. Besides, $\hat S$ identifies with
  $C(G)$ and $\Ff$ coincides with the Fourier transform of \cite[{\S}8,~n{\degree}1]
  {Bourbaki:lie_compact}. The finitely generated case corresponds to the case of compact
  Lie groups, which we address now.
  
  Let $G$ be a connected compact Lie group and choose a maximal torus $T\subset G$.
  Following~\cite{Bourbaki:lie_compact}, we denote by $X$ the dual group of $T$, by $R
  \subset X$ the set of roots of $G$, by $W$ the Weyl group of $(G,T)$ acting on $X$ and
  by $w_0$ the longest element of $W$.  We moreover choose a subset of positive roots $R_+
  \subset R$ and denote by $X_{++}$ the associated set of dominant weights. Taking the
  highest weight of irreducible representations defines a canonical identification between
  $\Irr\Cat$ and $X_{++}$.
  
  Let $||\,\cdot\,||$ be a $W$-invariant definite-positive quadratic form on
  $X\tens_\ZZ\RR$. Its restriction $l$ to $\Irr\Cat \simeq X_{++}$ is in fact a length
  function: in our identification, we have $\alpha \subset \beta\tens\beta'$
  $\Longrightarrow$ $\alpha \leq \beta + \beta'$, and because such inequalities are
  conserved by scalar product against dominant weights we obtain
  \begin{displaymath}
    ||\alpha||^2 \leq (\alpha | \beta + \beta') \leq 
    ||\beta + \beta'||^2 \leq (||\beta|| + ||\beta'||)^2 \text{.}
  \end{displaymath}
  Moreover, $||\bar\alpha|| = ||-w_0(\alpha)|| = ||\alpha||$ and $||1_\Cat|| = ||0_X|| =
  0$. We denote by $L$ the length function on $(S,\delta)$ associated to $l$.
  
  It is clear by definition of $L$ that $(s_n)$ has polynomial growth. Moreover $M_n$,
  which is the maximal dimension of the representations of length $n$, is also polynomialy
  growing by the dimension formula of H.~Weyl:
  \begin{displaymath}
    \dim \beta = \ts\prod_{\alpha\in R_+}\ds \frac {\langle \beta+\rho, K_\alpha\rangle}
    {\langle \rho, \alpha\rangle} \text{,}
  \end{displaymath}
  where $2\rho = \sum_{\alpha\in R_+} \alpha$ and the $\langle \,\cdot\,, K_\alpha
  \rangle$ are linear forms on $X$. As a result, duals of connected compact Lie groups
  have Property RD. In fact with little more work \cite[{\S}8,
  thm.~1]{Bourbaki:lie_compact} one can see that $\hat H_L^\infty$ coincides in this case
  with $C^\infty(G)$, which is evidently included in $C(G)$.
  
  On the other hand the $q$-deformations of simple compact Lie groups are usually
  non-unimodular and hence their duals do not have Property RD by
  Lemma~\ref{lem:growth_non_unimod} and Proposition~\ref{prp:growth}. This is for instance the
  case of the duals of the compact quantum groups $SU_q(N)$ for $q\in ]0,1[$, $N\geq 3$.  For
  $N=2$, the quantum group $SU_q(2)$ is defined for $q\in [-1,1] \setminus {0}$ and is
  non-unimodular for $|q|<1$. For $q = -1$ it has the same semi-ring of representations and the
  same dimension map as $SU(2)$, hence it has Property RD: this is a first non-commutative,
  non-cocommutative example.
\end{rk}

\subsection{The non-unimodular case}
\label{sec:non-unimod}

We have remarked in the previous section that Property RD is not conserved by non-unimodular
deformations. In this section we will see that Property RD is in fact uncompatible with
non-unimodularity. This will result from the geometric necessary condition
(\ref{eq:df_geom}) for Property RD which we will also use in the next section.

To move smoothly to the geometric point of view, let us use the notion of convolution of
elements of $S$ \cite{Woro:lorenz}. For $a$, $b\in\Ss$, we denote by $\Cc(a\tens b)$ the unique
element $c\in\Ss$ satisfying the relation $(ah_R \tens bh_R)\rond\delta = ch_R$, where $xh_R :=
h_R(\,\cdot\, x)$. This clearly defines a linear map $\Cc : \Ss\tens_{\mathrm{alg}}\Ss \to
\Ss$, which is related to the Fourier transform in the following way:
\begin{eqnarray*}
  \Ff(a)\Ff(b) &=& (\id\tens bh_R\tens ah_R)(V_{13}^* V_{12}^*) \\
  &=& (\id\tens bh_R\tens ah_R)(\id\tens\delta)(V^*)
  = \Ff(\Cc(b\tens a)) \text{.}
\end{eqnarray*}
In particular this yields the equality $\Ff(a)\Lambda(b) = \Lambda(\Cc(b\tens a))$.  Moreover it
is easy to check from the definition that
\begin{displaymath}
  \Cc(\delta(s)~ a\tens b) = s~ \Cc(a\tens b) ~~~\text{and}~~~
  \Cc(a\tens b~ \delta(s)) = \Cc(a\tens b)~ s \text{,}
\end{displaymath}
using for the second equality the KMS property of $h_R$. In particular we have $p_\gamma
\Cc(a\tens b) = \Cc(\delta(p_\gamma)(a\tens b)\delta(p_\gamma))$ and $p_\gamma\Cc(x\tens y)$ is
a multiple of $p_\gamma$ for any $x$, $y\in Z(\Ss)$, because it is then central.

\begin{lemma} \label{lem:norm_convol}
  Let $\gamma\subset \beta\tens\alpha$ be an inclusion without multiplicity, with
  $\alpha$, $\beta$, $\gamma \in \Irr\Cat$. Then we have
  \begin{displaymath}
    \forall x\in p_\beta S \tens p_\alpha S ~~~
    ||p_\gamma\Cc(x)||_2 = \sqrt{\frac{m_\beta m_\alpha}{m_\gamma}}~
    ||\delta(p_\gamma) x \delta(p_\gamma)||_2 \text{.}
  \end{displaymath}
\end{lemma}

\begin{proof}
  Let $\lambda\in\CC$ be the scalar such that $p_\gamma \Cc(p_\beta \tens p_\alpha) = \lambda
  p_\gamma$.  By the hypothesis on the inclusion $\gamma \subset \beta\tens\alpha$, for any
  $x\in p_\beta S \tens p_\alpha S$ there exists $y\in p_\gamma S$ such that
  $\delta(p_\gamma)x\delta(p_\gamma) = \delta(y) (p_\beta \tens p_\alpha)$. We then have
  $p_\gamma \Cc(x) = \Cc (\delta(p_\gamma)x\delta(p_\gamma)) = y p_\gamma \Cc(p_\beta \tens
  p_\alpha) = \lambda y$. As a result
  \begin{eqnarray*}
    ||p_\gamma \Cc(x)||_2^2 &=& \bar\lambda~ h_R(y^*p_\gamma \Cc(x))
    = \bar\lambda~ (h_R\tens h_R) (\delta(y^*)\delta(p_\gamma) x) \\
    &=& \bar\lambda~ (h_R\tens h_R) (\delta(p_\gamma) x^*\delta(p_\gamma) x)
    = \bar\lambda~ ||\delta(p_\gamma) x\delta(p_\gamma)||_2^2 \text{.}
  \end{eqnarray*}
  We obtain the desired value of $\lambda$ by the following computation:
  \begin{eqnarray*}
    \lambda\, m_\gamma^2 &=& \lambda\, h_R(p_\gamma) = h_R(p_\gamma \Cc(p_\beta\tens p_\alpha)) 
    = (h_R\tens h_R) (\delta(p_\gamma) (p_\beta\tens p_\alpha)) \\
    &=& m_\beta m_\alpha (\mathrm{Tr}\tens\mathrm{Tr}) ((F\tens F) \delta(p_\gamma)) 
    = m_\beta m_\alpha m_\gamma \text{.} 
  \end{eqnarray*} \par
\end{proof}

Let us now fix a central length $L$ on $(S,\delta)$ and denote by $l$ the
associated length function on $\Irr \Cat$. In view of the link between
convolution and Fourier transform, we have for any $a\in p_\alpha S$, $b\in
p_\beta S$:
\begin{equation*}
  ||p_\gamma \Ff(a) p_\beta \Lambda(b)|| = ||p_\gamma \Cc(b\tens a)||_2.
\end{equation*}
As a result Lemma~\ref{lem:norm_convol} gives a necessary condition for
caracterization~\ref{item:df_bloc} of Property RD to be fulfilled: there should
exist a polynomial $P \in \RR[X]$ such that one has, for any
inclusion $\gamma\subset \beta\tens\alpha$ without multiplicity and $a\in
p_\alpha S$, $b\in p_\beta S$:
\begin{equation} \label{eq:df_geom}
  ||\delta(p_\gamma) (b\tens a)\delta(p_\gamma)||_2 \leq 
  \sqrt{\frac{m_\gamma}{m_\beta m_\alpha}}~ P(\lfloor l(\alpha)\rfloor)~ 
  || b\tens a||_2.
\end{equation}
\pagebreak[3]

Let us observe that this last condition is of geometric nature: it concerns the
relative position in $H_\beta \tens H_\alpha$ of the cone of decomposable
tensors and of the $\gamma$-homogeneous subspace. To emphasize this point of
view we will now work in the identifications $p_\alpha S \simeq L(H_\alpha)$ and
use the twisted Hilbert-Schmidt norms $||x||_{HS}^2 = \Tr (F x^*x)$ on
$L(H_\alpha)$, which only differs from the $2$-norm on $p_\alpha S$ by a
coefficient $m_\alpha$. In particular, Inequality~(\ref{eq:df_geom}) can
equivalently be expressed with the twisted Hilbert-Schmidt norms.

Inequality~(\ref{eq:df_geom}) must in particular be satisfied for $\gamma =
1_\Cat$. In this case we have $\beta = \bar\alpha$, the inclusion is
automatically multiplicity free and it is realized by $t_{\bar\alpha}$, up to a
scalar. Since $||t_{\bar\alpha}|| = \sqrt{m_\alpha}$, we have
$||\delta(p_\gamma) (b\tens a)\delta(p_\gamma)||_{HS} = |t_{\bar\alpha}^*
(b\tens a) t_{\bar\alpha}| / m_\alpha$ and our necessary condition now reads:
$\forall \alpha\in\Irr\Cat$, $a\in L(H_\alpha)$, $b\in L(H_{\bar\alpha})$
\begin{equation} \label{eq:df_necess}
  |t_{\bar\alpha}^* (b\tens a) t_{\bar\alpha}| \leq 
  P(\lfloor l(\alpha)\rfloor) || b\tens a||_{HS} \text{.}
\end{equation}
The left-hand side has an even simpler expression, which comes from the definition of the
morphisms $t_{\bar\alpha}$ and holds in fact even if $\alpha$ is not irreducible. Let
$(e_i)$ be an orthonormal basis of $H_{\bar\alpha}$ and put $\bar a = j_{\bar\alpha}^* a
j_{\bar\alpha}$ for $a\in L(H_\alpha)$, for any such $a$ and $b\in L(H_{\bar\alpha})$ we
have
\begin{equation} \label{eq:dem_compr}
  t_{\bar\alpha}^* (b\tens a) t_{\bar\alpha} =   
  \ts\sum\ds (e_i | b e_j) (\bar e_i | a \bar e_j) = \Tr \bar a^* b \text{.}
\end{equation}

\begin{proposition} \label{prp:unimod_necess}
  Let $L$ be a central length on the Hopf \Cst algebra $(S,\delta)$ of a discrete quantum
  group. If $(S,\delta,L)$ has Property RD, then $(S,\delta)$ is unimodular.
\end{proposition}

\begin{proof}
  Let $\lambda \in \RR_+$ be an eigenvalue of $p_nF$, there exists a
  corresponding unit eigenvector $\xi\in H_{\bar\alpha}$ for some $\alpha\in
  \Irr\Cat$ with $l(\alpha) \in [n,n+1[$. We take $b=p_\xi$, the orthogonal
  projection onto $\CC\xi$, and $a = \bar p_\xi$. According
  to~(\ref{eq:dem_compr}), the left-hand side of~(\ref{eq:df_necess}) equals
  then $\Tr p_\xi = 1$.  On the other hand we have
  \begin{eqnarray*}
    ||p_\xi||_{HS}^2 &=& \Tr F p_\xi = \lambda \text{~~~and} \\
    ||\bar p_\xi||_{HS}^2 &=& \Tr F j^* p_\xi jj^* p_\xi j = 
    \Tr F^{-2}p_\xi F^{-1}p_\xi = \lambda^{-3} \text{.}
  \end{eqnarray*}
  Hence the condition~(\ref{eq:df_necess}) reads in this particular case $\lambda\leq P(n)$.
  Taking the supremum over the eigenvalues $\lambda$ shows that $||p_n F|| \leq P(n)$ for all
  $n$. This is impossible for a non-unimodular discrete quantum group by
  Lemma~\ref{lem:growth_non_unimod}.
\end{proof}

\subsection{The free quantum groups}
\label{sec:free_quantum}

We will mainly study the case of the duals of the orthogonal free quantum groups
$\hat S = A_o(Q)$, with $Q\in GL(N,\CC)$. Recall that $A_o(Q)$ is the \Cst
algebra generated by $N^2$ generators $u_{ij}$ and the relations making $U =
(u_{ij})$ unitary and $Q\bar U Q^{-1}$ equal to $U$
\cite{Wang:freeprod,DaeleWang:univ}. As usual, we will assume that $\bar QQ$ is
a scalar matrix, so that the fundamental corepresentation $U$ is irreducible.
When $N=2$, the discrete quantum groups in consideration correspond in fact to
the duals of the quantum groups $SU_q(2)$ \cite[section~5]{Banic:U(n)}, which we
have already studied. Moreover the dual of $A_o(Q)$ is unimodular \iff $Q$ is a
multiple of a unitary matrix, and up to an isomorphism one can then assume that
$Q = I_N$ or $Q = \big({0\atop-I_k}{I_k\atop0}\big)$
\cite{BichonRijdtVaes:ergodic}.

It is known that $\Cat$ identifies to the representation theory of $SU(2)$
\cite{Banic:O(n)_cras}: the irreducible representations $\alpha_n = \bar\alpha_n$ are indexed
by integers $n\in\NN$ in such a way that $\alpha_0 = 1_\Cat$ and $\alpha_1\tens\alpha_n \simeq
\alpha_{n-1} \oplus \alpha_{n+1}$. In particular we have $l(\alpha_n) = n$ with respect to
$\Dir = \{\alpha_1\}$ and the sequence of quantum dimensions $(m_n)_n$ satisfies the recursive
equation $m_1m_n = m_{n-1} + m_{n+1}$.  Moreover $m_0 = 1$ and $m_1$ is strictly greater than
$2$ when $N\geq 3$, hence in this case $m_n = r^n (1 - s^{n+1})/(1 - s)$ for some $r>1$ and $s
= r^{-2}<1$ \cite[lemma~2.1]{Vergnioux:cayley}.

\bigskip

In the case of the orthogonal free quantum groups, there is only one irreducible
representation of a given length, and consequently~(\ref{eq:df_geom}) is
equivalent to Property RD. More precisely we have by Lemma~\ref{lem:norm_convol}
\begin{equation*}
  ||p_l \Ff(a) p_k \Lambda(b)|| = ||p_l \Cc(b\tens a)||_2
  = \sqrt{\frac{m_k m_n}{m_l}}~ ||\delta(p_l) x \delta(p_l)||_2\text{,}
\end{equation*}
for $a \in p_n S$ and $b\in p_k S$. As a result
caracterization~\ref{item:df_bloc} of Property RD is satisfied \iff we have, for
any integers $k$, $l$, $n$ and every $a\in p_nS$, $b\in p_kS$:
\begin{equation} \label{eq:df_ao}
  ||\delta(p_l) (b\tens a) \delta(p_l)||_2
  \leq \sqrt{\frac{m_l}{m_n m_k}} P(n)~ ||b||_2 \, ||a||_2.
\end{equation}

Let us remark that the numerical coefficient in the right-hand side tends to
zero as $n$ goes to infinity. Hence the free quantum group under consideration
has Property RD \iff the cone of decomposable tensors in $H_k\tens H_n$ is
``asymptotically far'' from the subspace equivalent to $H_l$.  Like in
Section~\ref{sec:non-unimod}, we will study this geometric condition in the
identifications $p_nS \simeq L(H_n)$ and using the twisted Hilbert-Schmidt
norms.

For any representation $\alpha\in\Cat$ we have $\bar\alpha = \alpha$, recall
that we denote by $t_\alpha : \CC = H_{1_\Cat} \to H_\alpha \tens H_\alpha$ the
morphism associated to a normalized conjugation map on $H_\alpha$. For any
Hilbert spaces $H$, $H'$ we will also call $t_\alpha$ the map $\id\tens
t_\alpha\tens\id : H\tens H' \to H\tens H_\alpha\tens H_\alpha\tens H'$.  When
$\alpha = \alpha_k^{\tensexp n}$ we use the conjugation map
$j_{\alpha_k^{\tensexp n}} = \Sigma \rond (j_\alpha\tens j_{\alpha_k^{\tensexp
    n-1}})$ and the notation $t_k^n := t_\alpha$.

\begin{lemma} \label{lem:tech_ao}
  Let $L$ be the word length induced by $\Dir = \{\alpha_1\}$ on the dual of
  some $A_o(Q)$ with $N\geq 3$. Then $(S,\delta,L)$ has Property RD \iff there
  exists $P \in \RR[X]$ such that for all $k$, $l$, $n \in \NN$, $a\in L(H_n)$,
  $b\in L(H_k)$:
  \begin{displaymath}
    ||t_1^{q*} \delta(p_l)(b\tens a)\delta(p_l) t_1^q||_{HS} 
    \leq  P(n)~ ||b||_{HS} \, ||a||_{HS} \text{,}
  \end{displaymath}
  where we put $q = (n+k-l)/2$ and the norm in the left-hand side is the twisted
  Hilbert-Schmidt norm on $L(H_{k-q}\tens H_{n-q})$.
\end{lemma}

\begin{proof}
  We have $(b\tens a)\delta(p_l) t_1^q = (b\tens a) (p_k\tens
  p_n)t_1^q\delta(p_l)$. Since $(p_k\tens p_n)t_1^q : H_{k-q}\tens H_{n-q} \to
  H_k\tens H_n$ is a morphism, it is a multiple of an isometry on the highest
  homogeneous subspace $H_l \simeq \delta(p_l) (H_{k-q}\tens H_{n-q})$.  In view
  of the characterization of Property RD given by Inequality~(\ref{eq:df_ao}),
  the proof reduces to controlling the norm of $(p_k\tens p_n)t_1^q
  \delta(p_l)$. To do so we notice that
  \begin{displaymath}
    (p_k\tens p_n)t_1^q = (p_k\tens p_n)t_1 \rond (p_{k-1}\tens p_{n-1})t_1 \rond \cdots
    \rond (p_{k-q+1}\tens p_{n-q+1}) t_1\text{.}
  \end{displaymath}
  Now, the norm of each morphism $T_{p,p'} := (p_{p+1}\tens p_{p'+1}) t_1$ on the subspace
  of $H_p\tens H_{p'}$ equivalent to $H_l$ is given by \cite[prop.~2.3]{Vergnioux:cayley}:
  \begin{displaymath}
    ||T_{p,p'}\delta(p_l)||^2 = \frac{m_{p+1}}{m_p} \left(1- 
      \frac{m_{p-q}m_{p'-q-1}}{m_{p+1}m_{p'}} \right)  =: 
    \frac{m_{p+1}}{m_p} N_{p,p'}^l \text{,}
  \end{displaymath}
  with $q = \frac{p+p'-l}2$. So the numeric quantity we have to control is the following
  one:
  \begin{displaymath}
    \frac{m_k}{m_{k-q}} \sqrt{\frac{m_l}{m_n m_k}}  ~\cdot~
    (N_{k-q,n-q}^l \cdots N_{k-2,n-2}^l N_{k-1,n-1}^l) \text{.}
  \end{displaymath}
  Using the explicit expression of $m_p$ given at the beginning of the section, it is a
  boring but easy exercise to check that this quantity is bounded from above and from
  below by two non-zero constants independant of $k$, $l$, $n$. 
\end{proof}

\begin{theorem} \label{thm:RD_AO}
  Let $Q \in M_N(\CC)$ be an invertible matrix with $\bar QQ \in \CC I_N$ and $N\geq 3$. Then
  the dual of $A_o(Q)$ has Property RD with respect to the natural word length \iff $Q$ is
  unitary up to a scalar.
\end{theorem}

\begin{proof}
  We have already seen that the dual of $A_o(Q)$ does not have Property RD when $Q\notin
  \CC U(N)$, and hence we restrict to the case $Q\in U(N)$. Then $F=1$ and in particular
  there is no twisting in the Hilbert-Schmidt structures.  Let $(e_I)$ be a orthonormal
  basis of $L(H^{\tensexp q})$, ie a basis such that $\Tr(e_I^* e_J) = \delta_{I,J}$ for
  all $I$, $J$. We put $\bar e_I = j^* e_I j \in L(H^{\tensexp q})$, in our unimodular
  case $(\bar e_I)$ is again an orthonormal basis.
  
  Take $a \in L(H_n)$ and $b\in L(H_k)$. We consider $H_n$ (resp. $H_k$) as the highest
  homogeneous subspace of $H_1^{\tensexp n}$ (resp. $H_1^{\tensexp k}$) and we simply
  denote by $p_n$ (resp. $p_k$) the corresponding orthogonal projection.  Write $p_n a p_n
  = \sum \bar e_I\tens a_I$ and $p_k b p_k = \sum b_I\tens e_I$ with $a_I \in L(H_{n-q})$,
  $b_I \in L(H_{k-q})$. We have, using the identity~(\ref{eq:dem_compr}):
  \begin{eqnarray*}
    t_1^{q*}(b\tens a)t_1^q &=& (\id\tens t_1^{q*}\tens id) 
    (p_kbp_k\tens p_nap_n) (\id\tens t_1^{q}\tens id) \\ &=& 
    \ts\sum\ds b_I \otimes t_1^{q*}(e_I\tens\bar e_J)t_1^q \otimes a_J
    = \ts\sum\ds b_I \tens a_I \text{.}
  \end{eqnarray*}
  From this we get a first upper bound for the left-hand side of the inequality of
  Lemma~\ref{lem:tech_ao}: since $\delta(p_l) \in L(H_k\tens H_n)$ is an orthogonal
  projection,
  \begin{eqnarray*}
    ||\delta(p_l)t_1^{q*} (b\tens a) t_1^q \delta(p_l)||_{HS} \leq
    ||t_1^{q*} (b\tens a) t_1^q||_{HS} = ||\ts\sum\ds b_I \tens a_I||_{HS} \text{.}
  \end{eqnarray*}
  We then use the triangle and the Cauchy-Schwartz inequalities:
  \begin{displaymath}
    ||\ts\sum\ds b_I \tens a_I||_{HS}^2 \leq 
    \big(\ts\sum\ds ||b_I||_{HS} ||a_I||_{HS}\big)^2 
    \leq \ts\sum\ds ||b_I||_{HS}^2 \ts\sum\ds ||a_I||_{HS}^2 \text{.}
  \end{displaymath}
  Since $(e_I)$, $(\bar e_I)$ are orthonormal bases, we have $||a||_{HS}^2 = \sum
  ||a_I||_{HS}^2$ and $||b||_{HS}^2 = \sum ||b_I||_{HS}^2$. Hence the above estimate shows
  that the condition of Lemma~\ref{lem:tech_ao} for Property RD is fulfilled with $P=1$.
\end{proof}

We will now address briefly the case of the unitary free quantum groups $A_u(Q)$ with
$N\geq 3$. Its definition is similar to the one of $A_o(Q)$, using the relations that make
$U$ and $Q\bar U Q^{-1}$ unitary but not equal anymore. The corepresentation $U$ can then
be considered as a representation of $S$. It comes out that the result of
Theorem~\ref{thm:RD_AO} also holds for the duals of these quantum groups, the heuristic
reason being that $A_u(Q)$ is a mixing of the geometry of $A_o(Q)$ and of the
combinatorics of the free group $F_2$.
  
Let us recall the structure of $\Cat$ from \cite{Banic:U(n)}: $\Irr\Cat$ can be identified
with the free monoid on two generators $U$, $\bar U$ in such a way that the involutive
semi-ring structure is given by $\overline {\alpha U} = \bar U\bar\alpha$, $\overline
{U\alpha} = \bar\alpha\bar U$ and the recursive identities
\begin{eqnarray*}
  \alpha U\tens U\beta = \alpha UU\beta \text{,~~~} 
  \alpha U\tens \bar U\beta = \alpha U\bar U\beta \oplus \alpha\tens\beta \text{.}
\end{eqnarray*}
In particular the word length of the free monoid coincides with the word length on
$\Irr\Cat$ associated to the generating subset $\Dir = \{U, \bar U\}$.

\begin{theorem} \label{thm:RD_AU}
  The dual of $A_u(Q)$, with $Q\in GL_N(\CC)$ and $N\geq 3$, has Property RD with respect
  to the natural word length \iff $Q$ is unitary up to a scalar.
\end{theorem}

\begin{proof}
  Like in the orthogonal case $m_U$ is the geometric mean of $\Tr Q^*Q$ and $\Tr
  (Q^*Q)^{-1}$ and hence the dual of $A_u(Q)$ is unimodular \iff $Q\in\CC U(N)$. When this
  is not the case, we already know that Property RD is not satisfied. Moreover in the
  unimodular case $Q$ can be replaced with $I_N$ without changing the discrete quantum
  group under consideration.
  
  One can then follow the arguments of the orthogonal case to check that the necessary
  condition~(\ref{eq:df_geom}) is still satisfied. As a matter of fact
  Lemma~\ref{lem:tech_ao} relies on the technical result
  \cite[prop.~2.3]{Vergnioux:cayley} which holds in the unitary case for representations
  $\bar\beta$, $\alpha$ of the form $U\bar UU\cdots$, with respective lengths $k$, $n$.
  The reduction to this case is straightforward because $\alpha'UU\bar UU \cdots \simeq
  \alpha'U \tens U\bar UU \cdots$, compare \cite[rem.~6.4.2]{Vergnioux:cayley}. In this
  way one obtains for $A_u(I_N)$ the existence of a positive constant $C$ such that
  \begin{eqnarray} \label{eq:df_au}
    \forall~ \alpha, \beta, \gamma \in \Irr\Cat, ~~ a, b\in\Ss ~~~
    ||p_\gamma \Ff(p_\alpha a) \Lambda(p_\beta b)||_2 \leq 
    C~ ||p_\alpha a||_2 ||p_\beta b||_2 \text{.} \makebox[-1cm]{}
  \end{eqnarray}
  Because the combinatorics of the free monoid $\Irr\Cat$ is analogous (and in fact
  simpler for our purposes) to the one of the free group, one can show that this property
  is in fact sufficient, by adapting the ideas of \cite[lemma~1.3]{Haagerup:MAP} in the
  following way.
  
  Let us fix $n$, $k$, $l \in \NN$, $a\in p_nS$, $b\in p_kS$, and put $q = \frac
  {n-k-l}2$. For any $\alpha$, $\beta$, $\gamma\in\Irr\Cat$ such that $p_\gamma
  \Ff(p_\alpha a) \Lambda(p_\beta b)$ is non zero, there is an inclusion $\gamma \subset
  \beta\tens\alpha$, and hence we can write $\alpha = \tau\alpha'$, $\beta =
  \beta'\bar\tau$ and $\gamma = \beta'\alpha'$ with $l(\tau) = q$. Moreover, all triples
  $(\tau,\alpha',\beta')$ with $l(\tau) = q$, $l(\alpha') = n-q$ and $l(\beta') = k-q$ are
  obtained exactly once in this way. Using this ``change of indices'' one can prove
  Property RD via the last characterization of Definition~\ref{df:RD}. We compute indeed,
  for $a\in p_n\Ss$ and $b\in p_k\Ss$:
  \begin{eqnarray*}
    ||p_l \Ff(a) \Lambda(b)||^2 &=& \Big|\Big| \sum_{\scriptscriptstyle l(\gamma)=l}
    \sum_{l(\alpha)=n \atop l(\beta)=k} 
    p_\gamma \Ff(p_\alpha a) \Lambda(p_\beta b) \Big|\Big|^2 \\
    &=& \sum_{l(\alpha')=n-q \atop l(\beta')=k-q} 
    \Big|\Big| \sum_{\scriptscriptstyle l(\tau)=q} p_{\beta'\alpha'} 
    \Ff(p_{\tau\alpha'} a) \Lambda(p_{\beta'\bar\tau} b)\Big|\Big|^2 \text{,}
  \end{eqnarray*}
  because the projections $p_{\beta'\alpha'}$ are mutually orthogonal for different values of
  $(\beta',\alpha')$. We use then the triangle inequality, (\ref{eq:df_au}) and the
  Cauchy-Schwartz inequality:
  \begin{eqnarray*}
    ||p_l \Ff(a) \Lambda(b)||^2 &\!\!\leq\!\!& C^2 \sum_{\alpha',\, \beta'} \Big(
    \sum_{\tau} ||p_{\tau\alpha'} a||_2 ||p_{\beta'\bar\tau} b||_2 \Big)^2 \\
    &\!\!\leq\!\!& C^2 \sum_{\alpha',\, \beta'} \Big(\ts\sum\limits_\tau\ds 
    ||p_{\tau\alpha'} a||_2^2\Big) \Big({\ts\sum\limits_\tau\ds} 
    ||p_{\beta'\bar\tau} b||_2^2 \Big) = C^2 ||a||_2^2 ||b||_2^2 \text{.}
  \end{eqnarray*}  \par
\end{proof}
                       
\section{$K$-theory}

In this section we will check that the classical applications of Property RD to $K$-theory
still hold in the quantum case. More precisely, if $(S,\delta,L)$ has Property RD we will prove
that the subspaces $\hat H_L^\infty$ and $\hat H_L^s$, for $s$ big enough, are subalgebras of
$\hat S$ having the same $K$-theory as $\hat S$: compare \cite[thm.~A]{Jolissaint:RD_K_th} and
\cite[prop.~1.2]{Lafforgue:RD} respectively. Of course we restrict ourselves to unimodular
discrete quantum groups, since we have seen in Section~\ref{sec:non-unimod} that unimodularity
is necessary for Property RD to hold.

In fact following the methods of \cite{Ji:smooth_subalg} and \cite{Lafforgue:RD} this goes down
to establishing some norm inequalities, which we do at Proposition~\ref{prp:smooth_in_dom} and
Proposition~\ref{prp:laff_alg}. In particular in this section the difficulties of the quantum
generalization are only of technical nature. However, after having presented a definition and
quantum examples, it is also important to know that the applications are still working.

\subsection{The Fr\'echet algebra $\hat H_L^\infty$}

Let $L$ be a closed operator on $H$ admitting $\Hh$ as a core. For any bounded $x\in B(H)$, the
commutator $[L,x]$ is a priori an unbounded operator which needs not to be closable nor densely
defined. Let $\Dom D \subset B(H)$ be the subspace of operators $x$ such that $x\Hh \subset
\Dom L$ and $[L,x]$ is bounded on $\Hh$, and let us denote by $D(x) \in B(H)$ the closure of
$[L,x]$, for $x\in\Dom D$. This defines an unbounded linear map $D : \Dom D \to B(H)$. Because
$L$ is closed, it is a standard fact that $D$ is a closed derivation.

\begin{lemma}
  Let $L$ be a length on the Hopf \Cst algebra $(S,\delta)$ of a discrete quantum group.
  If $p_0$ has finite rank, we have $\hat S \cap \Dom D^k \subset \hat H_L^k$ (as
  subspaces of $H$).
\end{lemma}
  
\begin{proof} \label{lem:dom_is_smooth}
  Let $e \in H$ be a co-fixed unit vector for $V$ and $\Lambda : a \mapsto \Ff(a) e$ the
  associated GNS map for $h_R$. Since $\varepsilon(L) = 0$ we have $Le = 0$: one can indeed
  check that $\Lambda(p_\varepsilon)$ is a multiple of $e$, using the expression of $V$ on the
  image of $\Lambda\tens\Lambda$. In particular we have $D(x)e = Lxe$ for any $x\in\Dom D$.
  
  It is easy to check by induction that for any $x\in\Dom D^k$ we have $xe\in\Dom L^k$ and
  $D^k(x)e = L^k(xe)$: if this holds, let $x$ be in the domain of $D^{k+1}$, we have
  $D^k(x) \in \Dom D$ so that $D^k(x)e = L^k(xe) \in \Dom L$, and hence $xe \in \Dom
  L^{k+1}$. Moreover $D^{k+1}(x)e = D(D^k(x))e = LD^k(x)e = L^{k+1}xe$. In particular if
  $\hat a \in \hat S \cap \Dom D^k$, then $a e \in \Dom L^k$.
  
  Now observe that $L(1-p_0) \leq (1+L) (1-p_0) \leq 2L(1-p_0)$. Hence if $p_0$ has finite
  rank we have $\Dom L^k = \Dom (1+L)^k$. This concludes the proof because
  $\hat H_L^k = \Dom (1+L)^k$ by definition. 
\end{proof}

\begin{proposition} \label{prp:smooth_in_dom}
  Let $L$ be a central length on the Hopf \Cst algebra $(S,\delta)$ of a unimodular
  discrete quantum group. For any $k\in\NN$ we have $\hat\Ss \subset \Dom D^k$. Moreover
  if $(S,\delta,L)$ has Property RD with constants $s$, $C$ we have
  \begin{displaymath}
    \forall k\in\NN,~ \hat a\in\hat\Ss ~~~ ||D^k(\hat a)|| \leq 4C ||\hat a||_{2,s+k} \text{.}
  \end{displaymath}
\end{proposition}

\begin{proof}
  Proceeding by induction, we assume that the result holds for $k-1$. For $\hat
  a\in\hat\Ss$ it is clear that $D^{k-1}(\hat a)$ stabilizes $\Hh$, and hence $[L,
  D^{k-1}(\hat a)]$ is defined on $\Hh$ as well as its adjoint. We denote this operator by
  $D^k(\hat a)$ and we want to show that it is bounded. For $a\in\Ss$ it is easy to check
  by induction that
  \begin{displaymath}
    D^k(\Ff a) = (\id\tens h_R)(V^* (\delta(L) - L\tens 1)^k(1\tens a)) \text{.}
  \end{displaymath}
  Note that this is just the definition of $\Ff(a)$ for $k=0$, and use
  the identity $(L\tens 1) V^* = V^* \delta(L)$ to proceed to the
  induction. Using the expression of $V$ on the image of $\Lambda\tens\Lambda$ recalled in 
  Section~\ref{sec:notations} we obtain
  \begin{displaymath}
    D^k(\Ff a)^* \Lambda(b) = 
    (\Lambda\tens h_R)((1\tens a^*) (\delta(L) - L\tens 1)^k \delta(b)) \text{.}
  \end{displaymath}
  
  We first assume that $a$, $b \in \Ss$ are positive. By the first point of
  Lemma~\ref{lem:triangles} and since $\delta(b)$ commutes to $\delta(L) -
  L\tens 1$ on $\Hh\tens\Hh$ we have
  \begin{eqnarray*}
    -(1\tens L^k) \delta(b) \leq (\delta(L) - L\tens 1)^k \delta(b) 
    \leq (1\tens L^k) \delta(b) \text{.}
  \end{eqnarray*}
  Because $h_R$ is central, this yields
  \begin{displaymath}
    (\id\tens h_R)((1\tens a) (\delta(L) - L\tens 1)^k \delta(b)) \leq
    (\id\tens h_R)((1\tens aL^k) \delta(b)) 
  \end{displaymath}
  and similarly with the left inequality. But one can check that, for a central weight
  $\varphi$, inequalities of the form $-s \leq t\leq s$ with $t = t^*\in\Ss$ and
  $s\in\Ss_+$ imply the inequality $||t||_\varphi\leq ||s||_\varphi$ of the GNS norms. As
  a result we obtain
  \begin{displaymath}
    ||D^k(\Ff a)^* \Lambda(b)||_2 \leq ||\Ff(L^ka)^*\Lambda(b)||_2 \leq
    ||\Ff(L^ka)|| ~ ||b||_2 \text{.}
  \end{displaymath}
  
  This result is then easily generalized to any $b\in\Ss$, exactly like in the proof of
  Proposition~\ref{prp:growth}. Hence we have shown that $D^k(\Ff a)$ is bounded. Moreover if
  Property RD is satisfied we have the following estimate on its norm:
  \begin{displaymath}
    ||D^k(\Ff a)|| \leq 2~ ||\Ff(L^ka)|| \leq 2C~ ||L^ka||_{2,s} 
    \leq 2C~ ||a||_{2,s+k}\text{.}
  \end{displaymath}
  Again this can be generalized to any $a\in\Ss$ and we get the estimate of the statement,
  with $\hat a = \Ff(a)$. 
\end{proof}

\begin{corollary}
  Let $(S,\delta)$ be the Hopf \Cst algebra of a unimodular, finitely generated discrete
  quantum group with Property RD, and $L$ a word length on it. Then $\hat H_L^\infty$ is dense
  in $\hat S$ and coincides with $\bigcap \Dom D^k \cap \hat S$. In particular it is a dense
  subalgebra which is stable under holomorphic functional calculus in $\hat S$, and the
  inclusion $\hat H_L^\infty \subset \hat S$ induces isomorphisms in $K$-theory.
\end{corollary}

\begin{proof}
  Let $s$ be an exponent realizing Property RD, and $k\in\NN$. Let $\hat a$ be an element of
  $\hat H_L^{k+s}$: there exists a sequence $\hat a_n$ in $\hat\Ss$ converging to $\hat a$ in
  the $(2,k+s)$-norm. It is easy to check by induction that $\hat a\in\Dom D^l$ and $D^l \hat
  a_n \to D^l\hat a$ for $l= 0, \ldots, k$. As a matter of fact, $(\hat a_n)_n$ converges in
  particular in the $(2,l+s)$ norm, hence by Proposition~\ref{prp:smooth_in_dom} the sequence
  $D^l \hat a_n = D(D^{l-1}\hat a_n)$ has a limit in $B(H)$. Since $D^{l-1}\hat a_n \to
  D^{l-1}\hat a$ and $D$ is closed, this implies that $D^{l-1}\hat a \in \Dom D$ and $D^l \hat
  a_n \to D^l\hat a$.
  
  Hence we have proved that $\hat H_L^{k+s} \subset \Dom D^k \cap \hat S$ for all $k$.  Because
  we are using a word length, the hypothesis of Lemma~\ref{lem:dom_is_smooth} is satisfied and
  we also have $\Dom D^k \cap \hat S \subset \hat H_L^k$. This proves that $\hat H_L^\infty =
  \bigcap \Dom D^k \cap \hat S$. This subspace is dense because it contains $\hat\Ss$. It is
  then a general fact for closed derivations that $\hat H_L^\infty$ is a dense subalgebra
  stable under holomorphic functional calculus in $\hat S$, cf
  \cite[thm.~1.2]{Ji:smooth_subalg}. This implies in turn that the canonical inclusion induces
  isomorphisms in $K$-theory, cf e.g. \cite[prop.~8.14]{Valette:livre_BC} for a recent
  statement of this classical result. 
\end{proof}

\subsection{The Banach algebras $\hat H_L^s$}

We go on with the generalization of Lafforgue's result and start with a Lemma which is
proved using the same techniques as for Proposition~\ref{prp:smooth_in_dom}:

\begin{lemma} \label{lem:laff_ineq}
  Let $L$ be a central length on the Hopf \Cst algebra $(S,\delta)$ of a unimodular
  discrete quantum group. For any $(a_i) \in \Ss_+^n$ and $s, t\geq 0$ we have
  \begin{displaymath}
    ||\Ff(a_1) \cdots \Ff(a_n)||_{2,s+t} \leq n^t \ts\sum_i\ds
    ||\Ff(a_1) \cdots \Ff((1+L)^t a_i) \cdots \Ff(a_n)||_{2,s} \text{.}
  \end{displaymath}
\end{lemma}

\begin{proof}
  Using the identity $(\id\tens\kappa)(V) = V^*$ and the fact that $h_R\kappa^{-1}= h_R$
  in the unimodular case, one sees that $\Ff(a)^* = \Ff(\kappa(a^*))$. Since
  $\kappa(\Ss_+) = \Ss_+$ by unimodularity and $\kappa(L^*) = L$ by hypothesis, this
  allows to replace on both sides of the statement the first $n-1$ terms $\Ff(\cdot)$ by
  $\Ff(\cdot)^*$. In this way we avoid using $\kappa$ in the rest of the proof. We have
  indeed, using the identity $V_{12}V_{13}\cdots V_{1n} = (\id\tens\delta^{n-2})(V)$:
  \begin{eqnarray*}
    \Ff(a_1)^* \cdots \Ff(a_{n-1})^*\Ff(a_n) e &=&
    \Ff(a_1)^* \cdots \Ff(a_{n-1})^* \Lambda(a_n) \\ &=&
    (\Lambda\tens h^{\tensexp n-1}) ((1\tens a_{1}\tens\cdots\tens a_{n-1}) 
    \delta^{n-1}(a_n)) \text{.}
  \end{eqnarray*}
  
  Like in the classical case the proof relies on the following elementary inequality. For
  $t\geq 0$ the function $f_t : x \mapsto (1+x)^t$ is growing on $\RR_+$ and hence we have, for
  any $(x_i) \in \RR_+^{n}$:
  \begin{eqnarray*}
    (1+\ts\sum\ds x_i)^t &\leq& n^t (1+\ts\sum \frac {x_i}{n}\ds)^t \\
    &\leq& n^t (1+ \max x_i)^t \leq n^t \ts\sum\ds (1+x_i)^t \text{.}
  \end{eqnarray*}
  We apply this inequality to the following iterated version of the first point of
  Lemma~\ref{lem:triangles}, whose right-hand side is a sum of $n$ commuting terms:
  \begin{eqnarray*}
    && L\tens 1\tens\cdots\tens 1 \leq 1\tens L\tens 1\tens\cdots + \ldots + 
    1\tens \cdots\tens 1\tens L + \delta^{n-1}(L)  \\
    \Longrightarrow && f_t(L)\tens 1^{\tensexp n-1} \leq n^t \left(
      1\tens f_t(L)\tens 1^{\tensexp n-2} + \ldots + \delta^{n-1}(f_t(L)) \right) \text{.}
  \end{eqnarray*}
  
  This inequality can be multiplied by $(f_s(L)\tens 1^{\tensexp n-1})\delta^{n-1}(a_n)$, which
  is positive and commutes to all the terms. Applying moreover the positive functional
  $\id\tens h_Ra_1\tens \cdots\tens h_Ra_{n-1}$ we obtain
  \begin{eqnarray*}
    && \makebox[-.8cm]{} (f_{s+t}(L)\tens h^{\tensexp n-1}) 
    ((1\tens a_{1}\tens\cdots\tens a_{n-1}) \delta^{n-1}(a_n)) \leq \\
    && \makebox[.8cm]{} \leq n^t (f_s(L)\tens h^{\tensexp n-1}) 
    ((1\tens f_t(L)a_1\tens a_2\tens \cdots\tens a_{n-1}) \delta^{n-1}(a_n)) + \\
     && \makebox[1.2cm]{} + \ldots + n^t (f_s(L)\tens h^{\tensexp n-1}) 
     ((1\tens a_{1}\tens\cdots\tens a_{n-1}) \delta^{n-1}(f_t(L) a_n)) \text{.}
  \end{eqnarray*}
  Since $h_R$ is tracial, this inequality between positive elements of $\Ss$ implies the
  inequality of their $2$-norms and finally, by the triangle inequality, the inequality of 
  the statement.
\end{proof}

\begin{proposition} \label{prp:laff_alg}
  Let $L$ be a central length on the Hopf \Cst algebra $(S,\delta)$ of a unimodular discrete
  quantum group. Assume that $(S,\delta, L)$ has Property RD with exponent $s_0$. For any
  $s\geq s_0$, $\hat H_L^s$ is a Banach subalgebra of $\hat S$. Moreover for any $t\geq 0$
  there exists a constant $K_{s,t}>0$ such that
  \begin{displaymath}
    \forall \hat a\in\hat\Ss ~~ \exists C>0 ~~
    \forall n\in\NN~~ ||\hat a^n||_{2,s+t} \leq C K_{s,t}^{n} ||\hat a||_{2,s}^{n} \text{.}
  \end{displaymath}
\end{proposition}

\begin{proof}
  We first apply Lemma~\ref{lem:laff_ineq} with $n=2$ and $s=0$. Taking into account the
  unimodular fact that $||\hat a^*||_2 = ||\hat a||_2$ for $\hat a\in \hat S$, we obtain
  \begin{eqnarray*}
    ||\Ff(a_1) \Ff(a_2)||_{2,t} &\leq& 2^t ||\Ff(a_2)|| ~ ||\Ff((1+L)^ta_1)||_2 + \\
    && ~~~~ + 2^t ||\Ff(a_1)||~ ||\Ff((1+L)^t a_2)||_2 
  \end{eqnarray*}
  for $a_1$, $a_2 \in \Ss_+$. By definition we have $||\Ff((1+L)^ta_i)||_2 =
  ||\Ff(a_i)||_{2,t}$ and using moreover Property RD, which holds for any exponent $t\geq s_0$,
  we get
  \begin{displaymath}
    ||\Ff(a_1) \Ff(a_2)||_{2,t} \leq 2^{t+1} C ||\Ff(a_1)||_{2,t} ||\Ff(a_2)||_{2,t} \text{.}
  \end{displaymath}
  Now this extends to any $a$, $a' \in \Ss$ like in the proof of Proposition~\ref{prp:growth}.
  This proves that $||\hat a \hat a'||_{2,t} \leq 2^{t+3} C ||\hat a||_{2,t} ||\hat a'||_{2,t}$
  for any $\hat a$, $\hat a' \in \hat\Ss$ and hence the first statement.
  
  Fix $s\geq s_0$ and let $K_s\geq 1$ be such that $||\hat a \hat a'||_{2,s} \leq K_s
  ||\hat a||_{2,s} ||\hat a'||_{2,s}$ for any $\hat a$, $\hat a' \in \hat\Ss$. We get from
  Lemma~\ref{lem:laff_ineq}, for any $(a_i)\in \Ss_+^n$ and $t\geq 0$:
  \begin{eqnarray} \label{eq:dem_laff}
    && ||\Ff(a_1) \cdots \Ff(a_n)||_{2,s+t} \leq 
    \\ \nonumber && ~~~~ \leq n^t K_s^{n-1} \ts\sum_i\ds
    ||\Ff(a_1)||_{2,s} \cdots ||\Ff(a_i)||_{2,s+t} \cdots ||\Ff(a_n)||_{2,s} \text{.}
  \end{eqnarray}
  Now let $b$ be an element of $\Ss$ and write again the ``canonical decomposition'' $b =
  \sum_{k=0}^3 i^k b_k$, where the elements $b_k \in \Ss_+$ are the positive and negative parts
  of $\Re b$ and $\Im b$. Using the triangle and Cauchy-Schwartz inequalities as well as
  (\ref{eq:dem_laff}) we write
  \begin{eqnarray*}
    &&\!\!\!\! ||\Ff(b)^n||^2_{2,s+t} \leq (\ts\sum_{(k_j)_j}\ds 
      ||\Ff(b_{k_1}) \cdots \Ff(b_{k_n})||_{2,s+t} )^2 \\
    && ~~~~ \leq (\ts\sum_{i, (k_j)_j}\ds n^t K_s^n ||\Ff(b_{k_1})||_{2,s} \cdots 
      ||\Ff(b_{k_i})||_{2,s+t} \cdots ||\Ff(b_{k_n})||_{2,s} )^2 \\
    && ~~~~ \leq n^{2t}K_s^{2n} n4^n \ts\sum_{i, (k_j)_j}\ds ||\Ff(b_{k_1})||^2_{2,s} 
    \cdots ||\Ff(b_{k_i})||^2_{2,s+t} \cdots ||\Ff(b_{k_n})||^2_{2,s} \text{.}
  \end{eqnarray*}
  As already mentioned we have $||b||_{2,s}^2 = \sum ||b_k||_{2,s}^2$ because $h_R$ is tracial.
  Hence the last upper bound has the right form:
  \begin{eqnarray*}
    n^{2t+1} (2K_s)^{2n} \ts\sum_i\ds ||\Ff(b)||^2_{2,s} \cdots 
    ||\Ff(b)||^2_{2,s+t} \cdots ||\Ff(b)||^2_{2,s} = \makebox[2cm]{} \\ =
    n^{2t+2} (2K_s)^{2n} ||\Ff(b)||^2_{2,s+t}~ ||\Ff(b)||_{2,s}^{2n-2}  \text{.}
  \end{eqnarray*} \par
\end{proof}

\begin{corollary}
  Let $L$ be a central length on the Hopf \Cst algebra $(S,\delta)$ of a unimodular
  discrete quantum group. Assume that $(S,\delta, L)$ has Property RD with exponent $s_0$.
  Then the canonical inclusion of Banach algebras $\hat H_L^s \subset \hat S$ induces
  isomorphisms in $K$-theory for any $s\geq s_0$. 
\end{corollary}

\begin{proof}
  The proof goes exactly like in \cite[prop.~1.2]{Lafforgue:RD}. Denote by $\rho_s(\hat
  a)$ the spectral radius of $\hat a$ in $\hat H_L^s$. Taking the $n^{\rm th}$-root and
  letting $n$ go to infinity in the estimate of Proposition~\ref{prp:laff_alg}, we see
  that $\rho_{s+t}(\hat a) \leq K_{s,t} ||\hat a||_{2,s}$ for any $\hat a\in\hat\Ss$.
  Applying this to $\hat a^n$ and repeating the same process yields $\rho_{s+t}(\hat a)
  \leq \rho_s(\hat a)$, hence $\hat a$ has the same spectral radius in all the Banach
  algebras $\hat H_L^s$ with $s\geq s_0$.
  
  We use then an interpolation inequality for our Sobolev spaces which results, as in the
  classical case, from H\"older's inequality for the series \linebreak $\sum_\alpha
  {(1+l(\alpha))^{2s}} ||p_\alpha a||_2^2$. We obtain more precisely, for $s' > s \geq 0$:
  \begin{displaymath}
    \forall n\in\NN~~ ||\hat a^n|| \geq ||\hat a^n||_2 \geq ||\hat a^n||_{2,s}^{s'/(s'-s)}
    ||\hat a^n||_{2,s'}^{-s/(s'-s)} \text{,} 
  \end{displaymath}
  if $\hat a\in\hat\Ss$ is such that $\rho_{s'}(\hat a)\neq 0$. Again this yields an inequality
  between spectral radii, which reads $\rho_{\hat S}(\hat a) \geq \rho_s(\hat a) =
  \rho_{s'}(\hat a)$ when $s, s'\geq s_0$. Since $\hat H_L^s$ is dense in $\hat S$, this proves
  that it is stable under holomorphical calculus.
\end{proof}

\bibliographystyle{alpha} 
\small 
\bibliography{decay} 
\normalsize

\pagebreak
\thispagestyle{empty}

\begin{center}
    Roland {\sc Vergnioux} \\
    Mathematisches Institut (SFB 478) \\
    Westf\"alische Wilhelms-Universit\"at \\
    Einsteinstra\ss e 62 \\
    D--48149 M\"unster, Germany \\
    {\tt vergniou@math.uni-muenster.de}
\end{center}

\end{document}